\documentclass[a4paper,12pt,final]{amsart}
\usepackage{times,a4wide,mathrsfs,amssymb}

\newcommand{\C}{\mathbb{C}}

\newcommand{\QQ}{\mathbb{Q}}
\newcommand{\NN}{\mathbb{N}}
\newcommand{\PP}{\mathbb{P}}
\newcommand{\LL}{\mathbb{L}}

\newcommand{\MM}{\mathcal M}

\newcommand{\gr}{\hbox{Gr}}
\newcommand{\wt}{\widetilde}
\newcommand{\ima}{\hbox{Im}}
\newcommand{\rom}{\romannumeral}

\newcommand{\ide}{\hbox{id}}

\newtheorem{theorem}{Theorem}[section]
\newtheorem{claim}[theorem]{Claim}
\newtheorem{lemma}[theorem]{Lemma}

\newtheorem{proposition}[theorem]{Proposition}
\newtheorem{conjecture}[theorem]{Conjecture}
\newtheorem{remark}[theorem]{Remark}
\newtheorem{definition}[theorem]{Definition}
\newtheorem{convention}{Conventions}
\newtheorem{question}[theorem]{Question}

\newtheorem{nonumbering}{Theorem}

\newtheorem{nonumberingc}{Corollary}

\newtheorem{nonumberingt}{Acknowledgements}

\begin{document}

\author[Gilberto Bini]
{Gilberto Bini}

\address{Dipartimento di Matematica ``Federigo Enriques'', Universit\`a degli Studi di Milano, Via Saldini 50,
20133 Milano, ITALY.}
\email{gilberto.bini@unimi.it}

\author[Robert Laterveer]
{Robert Laterveer}

\address{Institut de Recherche Math\'ematique Avanc\'ee,
CNRS -- Universit\'e
de Strasbourg,\
7 Rue Ren\'e Des\-car\-tes, 67084 Strasbourg CEDEX,
FRANCE.}
\email{robert.laterveer@math.unistra.fr}

\author[Gianluca Pacienza]
{Gianluca Pacienza}

\address{Institut \'Elie Cartan de Lorraine \`a Nancy,
Universit\'e
de Lorraine, B.P. 70239, Vandoeuvre-l\`es-Nancy CEDEX,
FRANCE.}
\email{gianluca.pacienza@univ-lorraine.fr}

\title{Voisin's Conjecture for Zero--cycles on Calabi--Yau Varieties and their Mirrors}

\begin{abstract} We study a conjecture, due to Voisin, on 0-cycles on varieties with $p_g=1$. Using Kimura's finite dimensional motives and recent results of Vial's on the refined (Chow-)K\"unneth decomposition, we provide a general criterion for Calabi-Yau manifolds of dimension at most $5$ to verify Voisin's conjecture. We then check, using in most cases some cohomological computations on the
mirror partners, that the criterion can be successfully applied to various examples in each dimension up to $5$.
\end{abstract}



\maketitle

\section{Introduction}

For $X$ a smooth projective variety over $\C$, let $A^j(X)$ denote the Chow groups of codimension $j$ algebraic cycles on $X$ modulo rational equivalence.
Chow groups of cycles of codimension $>1$ are still mysterious. As an example, we recall the famous Bloch Conjecture, namely:

\begin{conjecture}[Bloch, \cite{B}]\label{bloch} Let $X$ be a smooth projective complex variety of dimension $n$. The following are equivalent:

\noindent
(\rom1) $A^n_{}(X)\cong \QQ$;

\noindent
(\rom2) the Hodge numbers $h^{j,0}(X)$ are $0$ for all $j>0$.
\end{conjecture}

The implication from (\rom1) to (\rom2) is actually a theorem \cite{BS}. The conjectural part is the implication from (\rom2) to (\rom1), which has been verified for surfaces not of general type in \cite{BKS}, but
it is wide open for surfaces of general type despite several significant cases have been dealt with over the years.
(see e.g. \cite{Bar, V12, BCGP, V8, PW}).

A natural next step is to consider varieties $X$ with geometric genus $p_g=1$. Here, the kernel  $A^n_{AJ}(X)$ of the Albanese map is huge; in a sense it is ``infinite--dimensional''  \cite{Mum} and \cite{Vbook}. Yet, this huge group should have controlled behaviour on the self--product $X \times X$, according to a conjecture due to Voisin, which is motivated by the Bloch--Beilinson conjectures (see \cite[Section 4.3.5.2]{Vo} for a detailed discussion).

\begin{conjecture}[\cite{V9},  see  \cite{Vo} Conjecture 4.37 for this precise form]\label{conjvois} Let $X$ be a smooth projective complex variety of dimension $n$ with $h^{j,0}(X)=0$ for $0<j<n$. The following are equivalent:

\noindent
(\rom1)
 For any zero--cycles $a,a^\prime\in A^n(X)$ of degree zero, we have
  \[ a\times a^\prime=(-1)^n a^\prime\times a\ \ \ \hbox{in}\ A^{2n}(X\times X)\ .\]
  (Here $a\times a^\prime$ is a short--hand for the cycle class $(p_1)^\ast(a)\cdot(p_2)^\ast(a^\prime)\in A^{2n}(X\times X)$, where $p_1, p_2$ denote projection on the first, resp. second factor.)

 \noindent
 (\rom2) the geometric genus $p_g(X)$ is $\le 1$.
 \end{conjecture}

 Again, the implication from (\rom1) to (\rom2) is actually a theorem (this can be proven \`a la Bloch--Srinivas \cite{BS}, see Lemma \ref{lem:warmup} below). The conjectural part is the implication from (\rom2) to (\rom1), which is still wide open for a general $K3$ surface (cf. \cite{V9}, \cite{moi}, \cite{desult}, \cite{todo}, \cite{epw} for some cases where this conjecture is verified).

 \vskip 0.3cm

In the present article we present a general criterion to check Voisin's conjecture (or a weak variant of it, cf. Theorem \ref{thm:weak}) for specific varieties (see section 1 for all the relevant definitions and explanations).

%This general result states that Voisin's conjecture is true for a Calabi--Yau manifold $X$ of dimension $n\leq 5$, provided $X$ verifies the standard Lefschetz conjecture, has finite--dimensional motive, has $\dim H^n_{tr}(X)=2$, and satisfies an extra condition in dimension 4 and 5

\begin{nonumbering}[=theorem \ref{main3}] Let $X$ be a smooth projective variety of dimension $n\le 5$ with $h^{i,0}(X)=0,\ 0<i<n$ and $p_g(X)=1$. Assume moreover that:

\noindent
(\rom1) $X$ is rationally dominated by a variety $X^\prime$ of dimension $n$, and $X^\prime$ has finite--dimensional motive and $B(X^\prime)$ is true;

%\noindent
%(\rom2) $B(X)$ is true;
%The Hodge conjecture is true for $X\times X$;

\noindent
(\rom2) %$\dim H^n_{tr}(X)=2$, i.e.
$X$ is $\wt N^1$-maximal.

\noindent
(\rom3) $\wt N^1 H^{i}(X)= H^{i}(X)$, for $0<i<n$.

Then conjecture \ref{conjvois} is true for $X$, i.e. any $a,a^\prime\in A^n_{hom}(X)$ verify
  \[ a\times a^\prime = (-1)^n a^\prime\times a\ \ \  \hbox{in}\ A^{2n}(X\times X)\ .\]
\end{nonumbering}

The proof of Theorem \ref{main3} relies, among other things, on results by Vial on the refined Chow-K\"unneth decomposition \cite{V4}, from which the hypotheses on $X'$ are thus inherited.

Our criterion can be effectively used to provide explicit examples in any dimension $\leq 5$.
Most of them are given by hypersurfaces of Fermat type in a (weighted) projective space (see Section 5 for all the examples).

The first and third hypotheses of our criterion hold for any Fermat hypersurface. As for the second, it seems the most delicate to verify in pratice.
In certain cases it is possible to check the second hypothesis by direct computation---e.g. for the Fermat sextic $X$ in $\mathbb P^5$, using results by Beauville, Movasati and the classical inductive structure of Fermat hypersurfaces (proposition \ref{sext}). Hence, we obtain the following explicit example:

\begin{nonumberingc}[=proposition \ref{sext}] Let $X\subset\PP^5(\C)$ be the sextic fourfold defined as
  \[ x_0^6 + \cdots +x_5^6=0\ .\]
  Then conjecture \ref{conjvois} is true for $X$, i.e. any $a,a^\prime\in A^4_{hom}(X)$ verify
  \[ a\times a^\prime =  a^\prime\times a\ \ \  \hbox{in}\ A^{8}(X\times X)\ .\]
\end{nonumberingc}

In other cases (for instance for the Fermat quintic 3-fold), despite the fact that the dimension of  $H^n(X)$ is quite large, it is possible to control the dimension of $H^n_{tr}(X)$ by passing to the mirror partner of $X$, which can be explicitly described in the Fermat case. Among other examples, we obtain in this way the $\wt N^1$-maximality and therefore Voisin's conjecture in the following case:

\begin{nonumberingc}[=proposition \ref{3fold}] Let $X\subset\PP(1^4,2)$ be the Calabi--Yau threefold defined as
  \[ x_0^6 + x_1^6 + x_2^6 + x_3^6 + x_4^3 =0\ .\]
  Then conjecture \ref{conjvois} is true for $X$, i.e. any $a,a^\prime\in A^3_{hom}(X)$ verify
  \[ a\times a^\prime =  -a^\prime\times a\ \ \  \hbox{in}\ A^{6}(X\times X)\ .\]
\end{nonumberingc}

% (see Lemma \ref{hat}). Then these 3 hypotheses hold simultaneously on each side of the mirror thanks to the existence of an isomorphism of Chow motives between $X$ and   $\hat{X}$ that can be established following \cite[Corollary 29]{desult} and its proof (see Proposition \ref{A_0} below).

%We would like to end the introduction with some words on possible further developments.
%It is well-known that the Fermat quintic threefold is  special member of a one-parameter family
%of quintics
%$$
% X_t :=\{ \sum_{i=0}^4 x_i^5 - t \prod_{i=0}^4x_i=0\},
%$$ known as the {\it  Dwork pencil}, for each member of which the mirror can be explicitely constructed, see \cite{CDGP}. Therefore
%Theorem \ref{main} and the intermediate results obtained to prove it lead naturally to some
% questions: can Conjecture \ref{conjvois} be proven for any smooth member $X_t$ of the  Dwork pencil and its mirror $\hat{X}_t$? Are the Chow motives (or at least the Chow groups of 0-cycles)
% of $X_t$ and   $\hat{X}_t$ isomorphic? Are these phenomena valid for other mirror pairs?
%  We hope to return to these questions in a near future.
\vskip0.6cm

\begin{convention}{\em In this note, the word {\sl variety\/} will refer to a reduced irreducible scheme of finite type over $\C$.

{\bf All Chow groups will be with rational coefficients:}
For a variety $X$, we will write $A_j(X)$ for the Chow group of $j$--dimensional cycles on $X$
with $\QQ$--coefficients.
For $X$ smooth of dimension $n$ the notations $A_j(X)$ and $A^{n-j}(X)$ will be used interchangeably.
%In the rare cases we will have something to say about Chow groups with integral coefficients, we will indicate this by writing $A^\ast X_{\ZZ}$.

The notations
%$A^j_{num}(X)$,
$A^j_{hom}(X)$ and $A^j_{AJ}(X)$ will be used to indicate the subgroups of
%numerically resp.
homologically, resp. Abel--Jacobi trivial cycles.
%For a morphism $f\colon X\to Y$, we will write $\Gamma_f\in A_\ast(X\times Y)$ for the graph of $f$.
The (contravariant) category of Chow motives (i.e., pure motives with respect to rational equivalence as in \cite{Sc}, \cite{MNP}) will be denoted $\MM_{\rm rat}$.
%The category of pure motives with respect to homological equivalence will be denoted $\MM_{\rm hom}$.

We will write $H^j(X)$ for singular cohomology $H^j(X,\QQ)$.
}
\end{convention}

%%%%%%%%%%%%%%%%%%%%%%%%%%%%%%%%
%
\section{Preliminaries}
%
%%%%%%%%%%%%%%%%%%%%%%%%%%%%%%%%

%%%
%
\subsection{Warm-up}
%
%%%
We begin with the following result
for which we could not find a reference in the literature, although it may be well-known to experts.
\begin{lemma}\label{lem:warmup}%[Voisin, Conjecture 4.37, \cite{V9}]\label{conjvois}
Let $X$ be a smooth projective complex variety of dimension $n$ with $h^{j,0}(X)=0$ for $0<j<n$. Consider the following conditions:

\noindent
(\rom1)
 For any zero--cycles $a,a^\prime\in A^n(X)$ of degree zero, we have
  \[ a\times a^\prime=(-1)^n a^\prime\times a\ \ \ \hbox{in}\ A^{2n}(X\times X)\ .\]
  (Here $a\times a^\prime$ is a short--hand for the cycle class $(p_1)^\ast(a)\cdot(p_2)^\ast(a^\prime)\in A^{2n}(X\times X)$, where $p_1, p_2$ denote projection on the first, resp. second factor.)

 \noindent
 (\rom2) the geometric genus $p_g(X)$ is $\le 1$.

 Then (i) implies (ii).
 \end{lemma}
\begin{proof}
This is a ``decomposition of the diagonal'' argument  \`a la Bloch-Srinivas:
%DETAILS TO BE FILLED OR CUT IT ALL. I WOULD SUGGEST CUTTING IT.
Let us define a correspondence
  \[ \pi:= \Delta_X - x\times X - X\times x\ \ \ \in A^n(X\times X)\ ,\]
  where $\Delta_X$ denotes the diagonal and $x\in X$. Next, we
 consider the correspondence
  \[ p:= ( \Delta_X - (-1)^n \Gamma_\iota)\circ (\pi\times\pi)\ \ \ \in A^{2n}\bigl( (X\times X)\times (X\times X)\bigr)\ ,\]
  where $\iota$ is the involution on $X\times X$ switching the two factors.

  Hypothesis (\rom1) implies that $p$ acts trivially on $0$--cycles of $X\times X$, i.e.
   \[  p_\ast A^{2n}(X\times X)=0\ .\]
   The Bloch--Srinivas argument \cite{BS} then implies there exists a rational equivalence
     \[ p= \gamma\ \ \ \hbox{in}\ A^{2n}\bigl( (X\times X)\times (X\times X)\bigr)\ ,\]
where $\gamma$ is a cycle supported on $X\times X\times D$, for some divisor $D\subset X\times X$. It follows that
  \[ \wedge^2 H^n(X) = p_\ast \bigl( H^n(X)\otimes H^n(X)\bigr)\ \ \ \subset \ H^{2n}(X\times X) \]
  is supported on the divisor $D$. In particular, we see that
    \[ \wedge^2 H^{n,0}(X,\C)\ \ \ \subset\ H^{2n,0}(X\times X,\C) \]
    is (supported on a divisor and hence) zero. This proves (\rom2).
\end{proof}

\begin{remark} We have actually proven more than the implication from (\rom1) to (\rom2). We have proven a special instance of the generalized Hodge conjecture: for any variety $X$ satisfying Lemma \ref{lem:warmup}, the sub Hodge structure
  \[ \wedge^2 H^n(X) \ \ \ \subset \ H^{2n}(X\times X) \]
  is supported on a divisor. This implication was already observed by Voisin \cite[Corollary 3.5.1]{Vo}.
  \end{remark}

%%%
%
\subsection{Finite--dimensional motives}
%
%%%

We refer to \cite{Kim}, \cite{An}, \cite{Iv}, \cite{J4}, \cite{MNP} for the definition of finite--dimensional motive.
An essential property of varieties with finite--dimensional motive is embodied by the nilpotence theorem.

\begin{theorem}[Kimura, Proposition 7.2, (ii), \cite{Kim}]\label{nilp} Let $X$ be a smooth projective variety of dimension $n$ with finite--dimensional motive. Let $\Gamma\in A^n(X\times X)_{\QQ}$ be a correspondence which is numerically trivial. Then there exists $N\in\NN$ such that
     \[ \Gamma^{\circ N}=0\ \ \ \ \in A^n(X\times X)_{\QQ}\ .\]
\end{theorem}

 Actually, the nilpotence property (for all powers of $X$) could serve as an alternative definition of finite--dimensional motive, as shown by a result of Jannsen \cite[Corollary 3.9]{J4}.
Conjecturally, any variety has finite--dimensional motive \cite{Kim}. We are still far from knowing this, but at least there are quite a few non--trivial examples.

\begin{remark}
{\em
The following varieties have finite--dimensional motive:  varieties dominated by products of curves (which is the case of the Fermat hypersurfaces) and abelian varieties \cite{Kim}, $K3$ surfaces with Picard number $19$ or $20$ \cite{P}, surfaces not of general type with vanishing geometric genus \cite[Theorem 2.11]{GP}, Godeaux surfaces \cite{GP}, certain surfaces of general type with $p_g=0$ \cite{V8}, \cite{BF},\cite{PW}, Hilbert schemes of surfaces known to have finite--dimensional motive \cite{CM}, generalized Kummer varieties \cite[Remark 2.9(\rom2)]{Xu},
 3--folds with nef tangent bundle \cite{Iy} (an alternative proof is given in \cite[Example 3.16]{V3}), 4--folds with nef tangent bundle \cite{Iy2}, log--homogeneous varieties in the sense of \cite{Br} (this follows from \cite[Theorem 4.4]{Iy2}), certain 3--folds of general type \cite[Section 8]{V5}, varieties of dimension $\le 3$ rationally dominated by products of curves \cite[Example 3.15]{V3}, %(which is the case of the mirror of the Fermat quintic),
 varieties $X$ with $A^i_{AJ}(X)_{}=0$ for all $i$ \cite[Theorem 4]{V2}, products of varieties with finite--dimensional motive \cite{Kim}.  \hfill $ \Box$

 }
\end{remark}

\begin{remark}
{\em
It is a (somewhat embarrassing) fact that all examples known so far of finite-dimensional motives happen to be in the tensor subcategory generated by Chow motives of curves (i.e., they are ``motives of abelian type'' in the sense of \cite{V3}).
That is, the finite--dimensionality conjecture is still unknown for any motive {\em not\/} generated by curves (on the other hand, there exist many motives not generated by curves, cf. \cite[7.6]{D}).   \hfill $ \Box$

}
\end{remark}

%%%
%
\subsection{Lefschetz standard conjecture and (co-)niveau filtration}
%
%%%

Let $X$ be a smooth projective variety of dimension $n$, and $h\in H^2(X,\QQ)$ the class of an ample line bundle. The hard Lefschetz theorem asserts that the map
  \[  L^{n-i}\colon H^i(X,\QQ)\to H^{2n-i}(X,\QQ)\]
  obtained by cupping with $h^{n-i}$ is an isomorphism, for any $i< n$. One of the standard conjectures, also known as Lefschetz standard conjecture $B(X)$, asserts that the inverse isomorphism is algebraic:

\begin{conjecture} Given a smooth projective variety $X$, the class $h\in H^2(X,\QQ)$
of an ample line bundle, and an integer $0\leq i<n$, the isomorphism
  \[  (L^{n-i})^{-1}\colon
  H^{2n-i}(X,\QQ)\stackrel{\cong}{\longrightarrow} H^i(X,\QQ)\]
  is induced by a correspondence.
 \end{conjecture}
 We recall the following filtration which, via Proposition \ref{N1max}, will play a central r\^ole in our criterion  (Theorem \ref{main3}) to check Conjecture \ref{conjvois}.
\begin{definition}[Coniveau filtration \cite{BO}]\label{con} Let $X$ be a quasi-projective variety. The {\em coniveau filtration\/} on cohomology and on homology is defined as
  \[\begin{split}   N^c H^i(X,\QQ)&= \sum \ima\bigl( H^i_Y(X,\QQ)\to H^i(X,\QQ)\bigr)\ ;\\
                           N^c H_i(X,\QQ)&=\sum \ima \bigl( H_i(Z,\QQ)\to H_i(X,\QQ)\bigr)\ ,\\
                           \end{split}\]
   where $Y$ (respectively $Z$) runs over codimension $\ge c$ (resp. dimension $\le i-c$) subvarieties of $X$, and $H^i_Y(X,\QQ)$ denotes the cohomology with support along $Y$.
 \end{definition}

  \begin{remark}\label{rmk:B(X)}
 {\em It is known that $B(X)$ holds for the following varieties: curves, surfaces, abelian varieties \cite{K0}, \cite{K}, threefolds not of general type \cite{Tan}, hyperk\"ahler varieties of
 $K3^{[n]}$--type \cite{ChM}, $n$--dimensional varieties $X$ which have $A_i(X)_{}$ supported on a subvariety of dimension $i+2$ for all $i\le{n-3\over 2}$ \cite[Theorem 7.1]{V}, $n$--dimensional varieties $X$ which have $H_i(X)=N^{\llcorner {i\over 2}\lrcorner}H_i(X)$ for all $i>n$ \cite[Theorem 4.2]{V2}, products and hyperplane sections of any of these \cite{K0}, \cite{K} (in particular it holds for projective hypersurfaces, a fact that we will use).

 For smooth projective varieties over $\C$, the standard conjecture $B(X)$ implies the standard conjecture $D(X)$, i.e homological and numerical equivalence coincide on $X$ and
 $X\times X$ \cite{K0}, \cite{K}.   \hfill $ \Box$

 }
 \end{remark}

Friedlander, and independently Vial, introduced the following variant of the coniveau filtration:

\begin{definition}[Niveau filtration \cite{Fr}, \cite{FM} \cite{V4}] Let $X$ be a smooth projective variety. The {\em niveau filtration} on homology is defined as
  \[ \wt{N}^j H_i(X)=\sum_{\Gamma\in A_{i-j}(Z\times X)_{}} \ima\bigl( H_{i-2j}(Z)\stackrel{\Gamma_\star}{\longrightarrow} H_i(X)\bigr)\ ,\]
  where the union runs over all smooth projective varieties $Z$ of dimension $i-2j$, and all correspondences $\Gamma\in A_{i-j}(Z\times X)_{}$.
  The niveau filtration on cohomology is defined as
  \[   \wt{N}^c H^iX:=   \wt{N}^{c-i+n} H_{2n-i}X\ .\]

\end{definition}

\begin{remark}{\em In \cite{Fr}, \cite{FM}, the filtration $\wt{N}^\ast$ is called the ``correspondence filtration'' rather than niveau filtration.
\hfill $ \Box$
}
\end{remark}

%\subsection{Coniveau and niveau filtration}

The relation between the standard conjecture $B(X)$ and the niveau and coniveau filtrations is made clear in the following.
\begin{remark}\label{is}
 {\em The niveau filtration is included in the coniveau filtration:
  \[ \wt{N}^j H^i(X)\subset N^j H^i(X)\ .\]
  These two filtrations are expected to coincide; indeed, one can show the two filtrations coincide if and only if the Lefschetz standard conjecture is true for all varieties \cite[Proposition 4.2]{Fr}, \cite[Proposition 1.1]{V4}.

  Using the truth of the Lefschetz standard conjecture in degree $\le 1$, it can be checked \cite[page 415 "Properties"]{V4} that the two filtrations coincide in a certain range:
  \[  \wt{N}^j H^i(X)= N^j H^iX\ \ \ \hbox{for\ all\ }j\ge {i-1\over 2} \ .\]

 In particular $\wt{N}^1 H^3(X)= N^1 H^3(X)$ and $\wt{N}^2 H^4(X)= N^2 H^4(X)$.
  \hfill $ \Box$
  }
  \end{remark}

The following ``refined K\"unneth decomposition'' and ``refined Chow--K\"unneth decomposition'' are very useful:

\begin{theorem}[Vial \cite{V4}]\label{vial} Let $X$ be a smooth projective variety of dimension $n\le 5$. Assume $B(X)$ holds.
There exists algebraic cycles $\pi_{i,j}$ on $X\times X$ and a decomposition of the diagonal
  \[ \Delta_X={\displaystyle \sum_{i,j} } \pi_{i,j}\ \ \ \hbox{in}\  H^{2n}(X\times X)\ ,\]
  where the $\pi_{i,j}$'s are mutually orthogonal idempotents. The correspondence $\pi_{i,j}$ acts on $H^\ast(X)$ as a projector on $\gr^j_{\wt{N}} H^i(X)$. Moreover, $\pi_{i,j}$ can be chosen to factor over a variety of dimension $i-2j$ (i.e., for each $\pi_{i,j}$ there exists a smooth projective variety $Z_{i,j}$ of dimension $i-2j$, and correspondences $\Gamma_{i,j}\in A^{n-j}(Z_{i,j}\times X), \Psi_{i,j}\in A^{i-j}(X\times Z_{i,j})$ such that $\pi_{i,j}=\Gamma_{i,j}\circ \Psi_{i,j}$ in $H^{2n}(X\times X)$).
\end{theorem}

\begin{proof} This is a special case of \cite[Theorem 1]{V4}. Indeed, as mentioned in loc. cit., varieties $X$ of dimension $\le 5$ such that $B(X)$ holds verify condition (*) of loc. cit.
\end{proof}
Under the extra hypothesis of the finite--dimensionality of the motive the conclusion can be proved at the level of Chow groups.
\begin{theorem}[Vial \cite{V4}]\label{vial2} Let $X$ be a smooth projective variety of dimension $n\le 5$. Assume $X$ has finite--dimensional motive and $B(X)$ holds.
There exists a decomposition of the diagonal
  \[ \Delta_X={\displaystyle \sum_{i,j} } \Pi_{i,j}\ \ \ \hbox{in}\  A^n(X\times X)\ ,\]
  where the $\Pi_{i,j}$'s are mutually orthogonal idempotents lifting the $\pi_{i,j}$ of Theorem \ref{vial}. Moreover, $\Pi_{i,j}$ can be chosen to factor over a variety of dimension $i-2j$ (i.e., for each $\Pi_{i,j}$ there exists a smooth projective variety $Z_{i,j}$ of dimension $i-2j$, and correspondences $\Gamma_{i,j}\in A^{n-j}(Z_{i,j}\times X), \Psi_{i,j}\in A^{i-j}(X\times Z_{i,j})$ such that $\Pi_{i,j}=\Gamma_{i,j}\circ \Psi_{i,j}$ in $A^n(X\times X)$).
\end{theorem}

\begin{proof} This is a special case of \cite[Theorem 2]{V4}. Indeed, $X$ as in theorem \ref{vial2} satisfies conditions (*) and (**) of loc. cit.
\end{proof}

\begin{remark}\label{rmk:alg}
{\em Let $X$ be as in Theorem \ref{vial}.
Notice that Conjecture $B(X)$ implies in particular that the $\pi_{i,j}$ are algebraic, cf. \cite[Theorem 4.1, item (3)]{K}.   \hfill $ \Box$
}
\end{remark}

\begin{remark}\label{t_n}
{\em
Let $X$ be as in Theorem \ref{vial2}. Then, as in \cite{desult}, one can define the ``most transcendental part'' of the motive of $X$ by setting
   \[ t_n(X):=(X,\Pi_{n,0},0)\ \ \ \in \MM_{\rm rat}\ .\]
The fact that $t_n(X)$ is well--defined up to isomorphism follows from \cite[Theorem 7.7.3]{KMP} and \cite[Proposition 1.8]{V4}.
In case $n=2$, $t_n(X)$ coincides with the ``transcendental part'' $t_2(X)$ constructed for any surface in \cite{KMP}.  \hfill $ \Box$

}
\end{remark}

%%%%%%%%%%%%%%%%%%%%%%%%
%
\section{$\wt N^1$--maximal varieties}
%
%%%%%%%%%%%%%%%%%%%%%%%%

Let $X$ be a smooth projective $n$-dimensional variety. Then $H^n(X)$ is a polarized Hodge structure, and the niveau $ N^1:=  N^1 H^n(X)$ is a Hodge substructure. If $X$ satisfies conjecture $B(X)$ it follows from the Hodge--Riemann bilinear relations (cf. for instance \cite[Theorem 2.22]{Vo}) that the Hodge substructure $N^1$ of the polarized Hodge structure $H^n(X,\QQ)$ induces a splitting
   \[ H^n(X,\QQ)=  N^1\oplus ( N^1)^{\perp}\ \]  (see \cite[Proposition 1.4 and Remark 1.5]{V4} for the details).

\begin{definition}  The ``transcendental cohomology'' is the orthogonal complement
   \[ H^n_{tr}(X):= (N^1)^{\perp}\ \ \subset\ H^n(X,\QQ)\ .\]

   \end{definition}
\begin{remark} \label{rmk:gr}
{\em  Note that $H^n_{tr}(X)$ is isomorphic to the graded piece $\gr^0_{ N^\bullet} H^n(X)$ (which is a priori only a quotient of $H^n(X)$).   \hfill $ \Box$

}
  \end{remark}
  One could also characterize $H^n_{tr}(X)$ by saying it is the smallest Hodge substructure $V\subset H^n(X,\QQ)$ for which $V_{\C}$ contains $H^{n,0}$.

\begin{proposition}\label{N1max} Let $X$ be a smooth projective n-fold. The following are equivalent:

\noindent
(\rom1) $\dim H^n_{tr}(X)=2 p_g(X)$;

\noindent
(\rom2) the subspace $H^{n,0}\oplus H^{0,n}\subset H^n(X,\C)$ is defined over $\QQ$;

\noindent
(\rom3) $\dim  N^1 H^n(X)= \sum_{i,j>0} h^{i,j}(X)$;

\noindent
(\rom4) the subspace $\oplus_{i,j>0} H^{i,j} \subset H^n(X,\C)$ is defined over $\QQ$.
\end{proposition}

\begin{proof} Obviously, (\rom1)$\Leftrightarrow$(\rom3). The equivalence (\rom2)$\Leftrightarrow$(\rom4) is obtained using the polarization on $H^n(X,\QQ)$. Indeed, suppose $V\subset H^n(X,\QQ)$ is a subspace such that $V_{\C}=H^{n,0}\oplus H^{0,n}$. Then $V\subset H^n(X,\QQ)$ is a Hodge substructure. As mentioned above, a Hodge substructure $V$ of the polarized Hodge structure $H^n(X,\QQ)$ induces a splitting
   \[ H^n(X,\QQ)= V\oplus V^\perp \]
   (cf. for instance \cite[Theorem 2.22]{Vo}). The subspace $V^\perp$ has $(V^\perp)_{\C}=\oplus_{i,j} H^{i,j}$.
  The rest is clear: (\rom1)$\Rightarrow$(\rom2) because (\rom1) forces $\bigl(H^n_{tr}(X)\bigr)_{\C}$ (which always contains $H^{n,0}\oplus H^{0,n}$) to be equal to  $H^{n,0}\oplus H^{0,n}$. Similarly, (\rom2)$\Rightarrow$(\rom1): if $V\subset H^n(X,\QQ)$ is such that $V_{\C}=H^{n,0}\oplus H^{0,n}$, then both $V$ and $H^n_{tr}(X)$ are the smallest Hodge substructure of $H^n(X,\QQ)$ containing $H^{n,0}$; as such, they are equal.
      \end{proof}

  \begin{definition} A smooth projective $n$-dimensional variety verifying the equivalent conditions of Proposition \ref{N1max} will be called {\em $ N^1$--maximal}.
  \end{definition}
  
  \begin{definition} A smooth projective $n$--dimensional variety $X$ will be called {\em $ \wt{N}^1$--maximal} if it is $N^1$--maximal and there is equality
    \[ N^1 H^n(X)=\wt{N}^1 H^n(X)\ .\]
    \end{definition}

 \begin{remark}\label{beau}
 {\em Proposition \ref{N1max} is inspired by \cite[Proposition 1]{Beau}, where a similar result is proven for surfaces. A surface with $\dim H^2_{tr}(S)=2 p_g(S)$ is called a $\rho$--maximal surface. 
 
 In dimension $n\le 3$, the notions of $N^1$--maximality and $\wt{N}^1$--maximality coincide, in view of remark \ref{is}.
   \hfill $ \Box$
}
  \end{remark}
  
 \begin{remark}\label{beau2}
 {\em While looking for examples of $N^1$-maximal Calabi-Yau 3folds we realised that the notion of $N^1$-maximality  was already considered (under a different name) in \cite[Remarks, p. 48, item 3)]{GM}, via the characterization (ii) of Proposition \ref{N1max}.
  }
  \end{remark}

As a consequence of Proposition \ref{N1max} we have the following nice property of $ N^1$--maximal $n$-folds $X$: they verify a strong (i.e., non--amended) version of the generalized Hodge conjecture:
  \[  H^n(X,\QQ)\cap F^1 = N^1 H^n(X,\QQ)\ \]

  where $F^1$ is the first piece of the Hodge filtration.

\section{A general result}

The following result gives sufficient conditions ensuring that a Calabi--Yau n-fold verifies Voisin's conjecture \ref{conjvois}:

\begin{theorem}\label{main3} Let $X$ be a smooth projective variety of dimension $n\le 5$ with $h^{i,0}(X)=0,\ 0<i<n$ and $p_g(X)=1$. Assume moreover that:

\noindent
(\rom1) $X$ is rationally dominated by a variety $X^\prime$ of dimension $n$, and $X^\prime$ has finite--dimensional motive and $B(X^\prime)$ is true;

%The Hodge conjecture is true for $X\times X$;

\noindent
(\rom2) %$\dim H^n_{tr}(X)=2$, i.e. 
$X$ is $\wt N^1$-maximal;

\noindent
(\rom3) $\wt N^1 H^{i}(X)= H^{i}(X)$, for $0<i<n$.

Then any $a,a^\prime\in A^n_{hom}(X)$ verify
  \[ a\times a^\prime = (-1)^n a^\prime\times a\ \ \  \hbox{in}\ A^{2n}(X\times X)\ .\]
\end{theorem}
\begin{remark}\label{rmk:dim1}
\em{You may notice that all hypotheses are satisfied in dimension 1.}
\end{remark}
Let
  \[ \iota\colon\ \ X\times X\ \to\ X\times X\]
  denote the involution exchanging the two factors.
  We consider the correspondence
   \[ \Lambda:= \frac{1}{2} (\Delta_{X\times X} +(-1)^{n+1} \Gamma_\iota)\ \ \ \in A^{2n}(X^4)\ ,\]
   where $\Delta_{X\times X}\subset X^4$ denotes the diagonal of $(X\times X)\times (X\times X)$, and $\Gamma_\iota$ denotes the graph of the involution $\iota$. Notice that $\Lambda$ is idempotent.
   To prove the Theorem \ref{main3} we must check that
   \[ \Lambda_\ast \ima \Bigl( A^n_{hom}(X)\otimes A^n_{hom}(X)\to A^{2n}(X\times X)\Bigr) =0\ .\]

   We need to modify $\Lambda$ a bit as follows.

    Let $\Psi\in A^n(X^\prime\times X)$ denote the closure of the graph of the dominant rational map $\psi$ from $X^\prime$ to $X$. We know that
     \begin{equation}\label{id} \Psi_\ast \Psi^\ast = d\cdot \ide\colon\ \ \ A^n(X)\ \to\ A^n(X)   \ ,\end{equation}
     where $d$ is the degree of $\Psi$.

Set $\Pi_{n,0}:= \frac{1}{d}\Psi \circ \Pi_{n,0}^{X'}\circ {}^t\Psi$
where $\Psi$ is as above and $\Pi_{n,0}^{X'}$ is given by Vial's result Theorem \ref{vial}, thanks to the finite dimensionality of the motive of $X'$ plus $B(X^\prime)$. Thanks to (\ref{id}) combined with the idempotence of  $\Pi_{n,0}^{X'}$, we have
\begin{equation}\label{eq:2}
 (\Pi_{n,0})_* \circ (\Pi_{n,0})_* = (\Pi_{n,0})_*:A^n(X)\ \to\ A^n(X).
\end{equation}
 Hence, up to dividing by a constant, we may assume that $(\Pi_{n,0})$
 acts as an idempotent on $0$-cycles on $X$.
    We finally introduce the correspondence
   \[ \Lambda_{tr}:=
   %\Lambda\circ(\Pi_{3,0}\times \Pi_{3,0})=
   \Lambda \circ (\Pi_{n,0}\times \Pi_{n,0})\ \ \ \in A^{2n}(X^4)\ ,\]
   where the $\Pi_{n,0}$ are as above (see \cite[Section 4.3.5.2]{Vo} for a similar construction). Note that $\Lambda_{tr}$ depends on the choice of $\Pi_{n,0}$. The key point is the following:

 \begin{claim}\label{claim:key}
    $\Lambda_{tr}$ acts as an idempotent on $0$-cycles, i.e.
    $$
     (\Lambda_{tr}{\circ}\Lambda_{tr})_* =  (\Lambda_{tr})_* :\ A_0 (X\times X)\to  A_0 (X\times X).
    $$
  \end{claim}

  \begin{proof}[Proof of Claim \ref{claim:key}]
 Notice that $\Lambda$ is an idempotent. Moreover
  by equation (\ref{eq:2}) also $\Pi_{n,0}$ acts as an idempotent on $0$-cycles.
  Write
\begin{eqnarray*}
%\begin{split}
   (\Lambda_{tr}{\circ}\Lambda_{tr})_*\!\!\!\!  &:=&\!\!\!\!    \frac{1}{4} [(\Delta_{X\times X} +(-1)^{n+1} \Gamma_\iota) \circ (\Pi_{n,0}\times \Pi_{n,0}) \circ  (\Delta_{X\times X} +(-1)^{n+1} \Gamma_\iota) \circ (\Pi_{n,0}\times \Pi_{n,0}) ]_*\\
 \!\!\!\!    &=&\!\! \!\!    [(\Lambda\circ \Lambda)\circ (\Pi_{n,0}\times \Pi_{n,0})\circ(\Pi_{n,0}\times \Pi_{n,0})]_*
 \\
  \!\! \!\! &=&\!\! \!\!  \Lambda_* (\Pi_{n,0}\times \Pi_{n,0})_*=(\Lambda_{tr})_*
%\end{split}
 \end{eqnarray*}

  where the second equality follows from the fact that $\Lambda$ and $\Pi_{n,0}$ commute (a fact that can either be checked by hand, or deduced fro the commutativity between $\Gamma_\iota$ and $\Pi_{n,0}$, which in turn follows from \cite[Lemma 3.4]{Kim}), while the third follows from equation (\ref{eq:2}).
 \end{proof}

  We will prove some intermediate results.
\begin{lemma}\label{lem:projector}  Set--up as in Theorem \ref{main3}.
The correspondence $\Lambda_{tr}$ acts on cohomology as a projector on the subspace
    \[\wedge^2 H^n_{tr}(X)\subset H^{2n}(X\times X)\ .\]
 \end{lemma}
 \begin{proof}
First we observe that $\Pi_{n,0}\times \Pi_{n,0}$ acts as projector onto $H^n_{tr}(X)\otimes H^n_{tr}(X)$. Next, for $\beta, \beta^\prime\in H^n_{tr}(X)$ we have
    \[  (\Delta_{X\times X} +\Gamma_\iota)_\ast (\beta\otimes \beta^\prime)= \beta\otimes\beta^\prime + (-1)^{n+1} \beta^\prime\otimes\beta\ \ \ \in H^{2n}(X\times X)\ .\]
    This shows that an element in $(\Lambda_{tr})_\ast H^\ast(X\times X)$ can be written as a sum of tensors of type
   \[ \beta\otimes\beta^\prime +(-1)^{n+1} \beta^\prime\otimes\beta\ ,\]
   with $\beta,\beta^\prime\in H^n_{tr}(X)$.
Since the cup--product map
$$
 H^n(X)\otimes H^n(X)\to H^{2n}(X)
$$
is $(-1)^{n^2}$-commutative, tensors of this type correspond exactly to elements
      \[  \bigl\{ b\in \ima \bigl( H^n_{tr}(X)\otimes H^n_{tr}(X)\to H^{2n}(X\times X)\bigr)\ \vert\ \iota_\ast(b)=-b\bigr\}\ .\]
      Thus,
  \[     (\Lambda_{tr})_\ast H^\ast(X\times X)\cong \wedge^2 H^n_{tr}(X)\ \subset\ H^{2n}(X\times X)\ .\]
\end{proof}

\begin{remark}
{\em Just to fix ideas, let us suppose for a moment that $X$ and $X^\prime$ coincide, so that $\Pi_{n,0}$ (and hence $\Lambda_{tr}$) is idempotent.
In this case, $\Lambda_{tr}$ defines the Chow motive
    \[ \hbox{Sym}^2 t_n(X)\ \ \ \in \MM_{\rm rat}\ \]
  in the language of \cite[Definition 3.5]{Kim}, where $t_n(X)$ is the ``transcendental motive'' $(X,\Pi_{n,0},0)$ as in Remark \ref{t_n}.  \hfill $ \Box$

}
\end{remark}

  The next lemma ensures that $\Lambda$ and $\Lambda_{tr}$ have the same action on the $0$--cycles that we are interested in. This is the only place in the proof where we need the full force of hypothesis (\rom3).

  \begin{lemma}\label{same} Set--up as in Theorem \ref{main3}. Let
%  $$
%    \ima^{2j}_{AJ}:=\ima \Bigl( A^j_{AJ}(X)\otimes A^j_{AJ}(X)\xrightarrow{\times} A^{2j}(X\times X)\Bigr)\ \subset\   A^{2j}(X\times X)
 % $$
% and
 $$
  %\ima^{}_{}
   A^{(n,n)} :=\ima \Bigl( A^n_{}(X)\otimes A^n_{}(X)\xrightarrow{\times} A^{2n}(X\times X)\Bigr)\ \subset\   A^{2n}(X\times X)
 $$
 and let
  \[ A^{(2,2)} :=\ima \Bigl( A^2_{AJ}(X)\otimes A^2_{AJ}(X)\xrightarrow{\times} A^{4}(X\times X)\Bigr)\ \subset\   A^{4}(X\times X) \]
 (where $\times$ denotes the map sending $a\otimes a^\prime$ to $a\times a^\prime$).

  Then for any choice of $\Pi_{n,0}$ as in Theorem \ref{vial}, we have
    \[ {(\Lambda_{tr})_\ast}_{| A^{(n,n)}_{}}= {\Lambda_\ast}_{| A^{(n,n)}_{}} ,\]
   % In particular,
   and
        \[ {(\Lambda_{tr})_\ast}_{| A^{(2,2)}_{}}= {\Lambda_\ast}_{| A^{(2,2)}_{}}  .\]

        \end{lemma}

 \begin{proof}
 %CHANGE PROOF IN DIM 4 AND 5.
  The point is that according to Theorem \ref{vial}, there is a decomposition
   \[ \Delta_X=\Pi_{n,0} +{\displaystyle\sum_{(i,j)\not=(n,0)}} \Pi_{i,j}\ \ \ \hbox{in}\ A^n(X\times X)\ .\]
   We claim that the components $\Pi_{i,j}$ with $(i,j)\not=(n,0)$ do not act on $A^n_{}(X)$:
   \[ (\Pi_{i,j})_\ast A^n_{}(X)=0\ \ \ \hbox{for\ all\ }(i,j)\not=(n,0)\ .\]
   Indeed, $\Pi_{i,j}$ may be chosen to factor over a variety $Z$ of dimension $i-2j$ (by Theorem \ref{vial}). Hence, the action of $\Pi_{i,j}$ on $A^n_{}(X)$ factors as follows:
   \[ (\Pi_{i,j})_\ast\colon\ \ A^n_{}(X)\ \to\ A^{i-j}_{}(Z)\ \to\ A^j(X)\ ,\]
 Now, our hypotheses imply that any $\Pi_{i,j}$ different from $\Pi_{n,0}$ has $j>0$.  Thus, the group in the middle is $0$ (for dimension reasons), and the claim is proven.

    We now consider the diagonal $\Delta_{X\times X}$ of the self--product $X\times X$. There is a decomposition
   \[ \Delta_{X\times X}= {\displaystyle\sum_{i,j,i^\prime,j^\prime}} \Pi_{i,j}\times \Pi_{i^\prime,j^\prime}\ \ \ \hbox{in}\ A^{2n}(X^4)\ .\]
   Let $a,a^\prime\in A^n_{}(X)$.
   Using the claim, we find that
    \[  (\Pi_{i,j}\times \Pi_{i^\prime,j^\prime})_\ast (a\times a^\prime)= (\Pi_{i,j})_\ast(a)\times (\Pi_{i^\prime,j^\prime})_\ast (a^\prime)= 0\ \ \ \hbox{for\ }(i,j,i^\prime,j^\prime)\not=(n,0,n,0)\ .\]
    It follows that
    \[ a\times a^\prime=(\Delta_{X\times X})_\ast (a\times a^\prime)= (\Pi_{n,0}\times \Pi_{n,0})_\ast (a\times a^\prime)\ \ \ \hbox{in}\ A^{2n}(X\times X)\ ,\]
    which proves the $A^{(n,n)}$ statement.

    The second statement of lemma \ref{same} is proven similarly:
    we claim that the components $\Pi_{i,j}$ with $(i,j)\not=(n,0)$ do not act on $A^2_{AJ}(X)$. This claim follows from the factorization
      \[  (\Pi_{i,j})_\ast\colon\ \ A^2_{AJ}(X)\ \to\ A^{2+i-j-n}_{AJ}(Z)\ \to\ A^2_{2}(X)\ ,\]
      where $\dim Z=i-2j$ (one readily checks that for $j>0$, the middle group vanishes in all cases).
      \end{proof}

 We now use the hypothesis that $\dim H^n_{tr}(X)=2$ and verify that the Hodge conjecture holds for the one--dimensional subspace $\wedge^2 H^n_{tr}(X)$.
    % \[  \wedge^2 H^3_{tr}(X)\subset H^{3,3}(X\times X,\QQ)\ ,\]
     %and that

     \begin{lemma}\label{one} Set--up as in Theorem \ref{main3}.
    \begin{itemize}

    \item[(i)] The subspace $( H^n_{tr}(X)\otimes H^n_{tr}(X))\cap F^1\subset H^{2n}(X\times X)$ has dimension $1$ and is generated by the cycle $\pi_{n,0}\in A^n(X\times X)$ given by Theorem \ref{vial}.

\item[(ii)]
     $  \wedge^2 H^n_{tr}(X) = \QQ[ \Pi_{n,0}]$ in $H^{2n}(X\times X)$.
     \end{itemize}

     \end{lemma}

     \begin{proof} Set $V:=H^n_{tr}(X)$.

     (i) We first note that, thanks to the hypothesis of $\wt N^1$-maximality and Proposition \ref{N1max}, we have  $V_{\C}=H^{n,0}\oplus H^{0,n}$. Hence
    \[ (V\otimes V)_{\C}=V_{\C}\otimes V_{\C}\subset H^{2n,0}\oplus H^{n,n}\oplus H^{0,2n}\ .\]
    It follows that
    \[ (V\otimes V)\cap F^1 = (V\otimes V)\cap F^n\ .\]

    The complex vector space
   \[ F^n (V_\C \otimes V_\C)= (H^{0,n}(X)\otimes H^{n,0}(X)) \oplus ( H^{n,0}(X)\otimes H^{0,n}(X) )\]
   is $2$--dimensional, with generators $c,d$ such that $c=\bar{d}$. Let
     \[ a\in (V\otimes V)\cap F^n\ ,\]
     i.e. $a$ is such that the complexification $a_\C\in H^{2n}(X\times X,\C)$ can be written
   \[   a_\C= \lambda c +\mu \bar{c}\ .\]
   But the class $a_\C$, coming from rational cohomology, is invariant under conjugation, so that $\lambda=\mu$, i.e.
   \[ \dim (V\otimes V)\cap F^n =1\ .\]
   Let $\pi_{n,0}$ be the cycle given by Theorem \ref{vial}. The class of $\pi_{n,0}$ in $H^{2n}(X\times X)$ lies in $V\otimes V$ because $\pi_{n,0}$ is a projector on $V$, i.e.
    $V=(\pi_{n,0})_\ast H^\ast(X)$.
   As the class $\pi_{n,0}\in H^{2n}(X\times X)$ is non--zero (for otherwise $H^n_{tr}(X)=0$ and $p_g(X)=0$), $\pi_{n,0}$ generates the one--dimensional subspace $(V\otimes V)\cap F^n$.

     (ii) Since $p_g(X)=1$, we have
       \[ \wedge^2 H^n_{tr}(X)\ \subset\ H^{2n}(X\times X)\cap F^1\ .\]
     It follows that
      \[ \wedge^2 H^n_{tr}(X)=\bigl(\wedge^2 H^n_{tr}(X)\bigr)\cap F^1 \ \subset\     \bigl( H^n_{tr}(X)\otimes H^n_{tr}(X)\bigr)\cap F^1\ .\]
      By item (i) we have that
  $\bigl( H^n_{tr}(X)\otimes H^n_{tr}(X)\bigr)\cap F^1$      is one--dimensional with generator $\pi_{n,0}$ and the conclusion follows.
      \end{proof}
 We now have all the ingredients for the:
\begin{proof}[Proof of Theorem \ref{main3}] (For a related conjecture, the argument that follows was hinted at in \cite[Remark 35]{desult}.)

   Consider the correspondence $\Lambda_{tr}\in A^{2n}(X^4)$. By Lemma \ref{lem:projector} it acts on $H^\ast(X\times X)$ by projecting onto $\wedge^2 H^n_{tr}(X)\subset H^{2n}(X\times X)$.
   This implies there is a containment
     \[ \Lambda_{tr} \in \bigl(\wedge^2 H^n_{tr}(X)\bigr)\otimes  \bigl(\wedge^2 H^n_{tr}(X)\bigr)\ \ \subset\ H^{4n}(X^4)\ .\]
     By Lemma \ref{one}, the subspace $\wedge^2 H^n_{tr}(X)$ is one--dimensional and generated by a cycle $\Pi_{n,0}\in A^n(X\times X)$. It follows there is a codimension $n$ subvariety $P\subset X\times X$ (the support of $\Pi_{n,0}$) such that
      \[ \Lambda_{tr}=\gamma\ \ \ \hbox{in}\ H^{4n}(X^4)\ ,\]
      where $\gamma$ is a cycle supported on $P\times P\subset X^4$.
   In other words, we have
    \[ \Lambda_{tr}-\gamma\ \ \ \in A^{2n}_{hom}(X^4)\ .\]

    Recall that $\Psi\in A^n(X^\prime\times X)$ denotes the closure of the graph of the dominant rational map $\psi$ from $X^\prime$ to $X$.
    The correspondence
  \[ \Gamma:= ({}^t \Psi\times {}^t \Psi)\circ (      \Lambda_{tr}-\gamma)\circ (\Psi\times\Psi)\ \ \ \in A^{2n}((X^\prime)^4) \]
  is homologically trivial (because the factor in the middle is homologically trivial). Using finite--dimensionality
and Theorem \ref{nilp}, we know there exists $N\in\NN$ such that
   \[ \Gamma^{\circ N}=0\ \ \ \hbox{in}\ A^{2n}((X^\prime)^4) \ .\]
   In particular, this implies that
   \[  (\Psi\times \Psi) \circ   \Gamma^{\circ N}\circ  ({}^t \Psi\times {}^t \Psi)=0\ \ \ \hbox{in}\ A^{2n}(X^4) \ .\]
    Developing this expression, and applying the result to $0$--cycles, and repeatedly using relation (\ref{id}), we obtain
      \[  \bigl((\Lambda_{tr})^{\circ N}\bigr){}_\ast = \bigl(Q_1 + Q_2 +\cdots + Q_N\bigr){}_\ast \colon \ \ \ A^{2n}(X\times X)\ \to\ A^{2n}(X\times X)\ ,\]
      where each $Q_j$ is a composition of $\Lambda_{tr}$ and $\gamma$ in which $\gamma$ occurs at least once.
      Since $\Lambda_{tr}$ is an idempotent, this simplifies to
      \[  (\Lambda_{tr}){}_\ast = \bigl(Q_1 + Q_2 +\cdots + Q_N\bigr){}_\ast\colon \ \ \ A^{2n}(X\times X)\ \to\ A^{2n}(X\times X)\ .\]
      The correspondence $\gamma$ acts trivially on $A^{2n}(X\times X)$ for dimension reasons, and so the $Q_j$ likewise act trivially on $A^{2n}(X\times X)$. It follows that
      \[ (\Lambda_{tr})_\ast =\bigl( Q_1 +\cdots + Q_N\bigr){}_\ast =0\colon\ \ A^{2n}(X\times X)\ \to\ A^{2n}(X\times X)\ .\]

      By Lemma \ref{same} this ends the proof of Theorem \ref{main3}.
%
%      we find
%      \[ \Lambda_\ast=(\Lambda_{tr})_\ast =0\colon\ \ \ima \Bigl( A^3_{hom}(X)\otimes A^3_{hom}(X)\to A^6(X\times X)\Bigr)\ \to\   A^6(X\times X) \ .\]
%      This means that for any $a,a^\prime\in A^3_{hom}(X)$, we have
%      \[ \Lambda_\ast (a\times a^\prime) = a\times a^\prime + a^\prime\times a =0\ \ \ \hbox{in}\ A^6(X\times X)\ .\]
       \end{proof}

\begin{remark}
{\em
The above proof is somehow indirect as we are able to prove the statement for the auxiliary correspondence $\Lambda_{tr}$, and then check that its action on $A^n_{hom}(X)\otimes A^n_{hom}(X)$ coincides with that of $\Lambda$.  \hfill $ \Box$

}
\end{remark}

\begin{remark}
{\em Hypothesis (\rom1) of theorem \ref{main3} may be weakened as follows: it suffices that there exists $X^\prime$ of dimension $\le 5$ such that $X^\prime$ has finite--dimensional motive and $B(X^\prime)$ is true, and there exists a correspondence from $X^\prime$ to $X$ inducing a surjection
  \[ A^i(X^\prime)\ \twoheadrightarrow\ A_0(X)\ .\]
  The argument is similar.
  \hfill $ \Box$

  }
  \end{remark}

\begin{remark}
{\em
We have seen (Remark \ref{beau}) that n-dimensional manifolds with $\dim H^n_{tr}(X)=2$ are a higher--dimensional analogue of $\rho$--maximal surfaces. In \cite[Proposition 5]{moi}, it is shown that surfaces $S$ with finite--dimensional motive and $\dim H^2_{tr}(S)=2$ (i.e. $p_g=1$ and $S$ is $\rho$--maximal) verify Voisin's conjecture. Theorem \ref{main3} is a higher--dimensional analogue of this result.  \hfill $ \Box$

}
\end{remark}

\begin{remark}
{\em Following Voisin's approach \cite{V9} one can extend the analysis above to $0$-cycles on higher products of $X$ with itself. In this direction we get the following.

\begin{theorem}\label{thm:weak}
\label{main4}
Let $X$ be a smooth projective variety of dimension $n$ less than or equal to $5$. Assume further that $h^{i,0}(X)=0$ for $0 < i <n$ and $p_g(X)\le 2$. Suppose moreover that
\begin{enumerate}
\item $X$ is rationally dominated by a variety $X^\prime$, and $X^\prime$ has finite dimensional motive and $B(X^\prime)$ is true;
%\item $B(X)$ is true;
\item the dimension of $H^n_{tr}(X)$ is at most $4$;
\item $\tilde{N}^1H^i(X)=H^i(X)$ for $0 < i < n$.
\end {enumerate}

Then any $a_1, a_2, a_3, a_4 \in A^n_{hom}(X)$ verify
$$
\sum_{\sigma \in {\mathfrak S}_4} \varepsilon(\sigma) \sigma^*(a_1 \times a_2 \times a_3 \times a_4)=0 \quad \hbox{in} \quad A^{4n}(X \times X \times X \times X).
$$
\end {theorem}

\begin{proof}
The proof closely follows that of Theorem \ref{main3}. In that situation, we took into account $\Lambda^2\left( H^n_{tr}(X)\right)$ and, after that, described a generator of it via an explicit cycle that is induced by a correspondence. In this situation, it is possible to give a generator of the $1$-dimensional space $\Lambda^4\left( H^n_{tr}(X)\right)$. The rest of the proof is similar to that in Theorem \ref{main3}.
\end{proof}
}
\end{remark}

Conjecturally, any variety $X$ with $h^{2,0}(X)=0$ should have $A^2_{AJ}(X)=0$ (this would follow from the Bloch--Beilinson conjectures, or a strong form of Murre's conjectures). We cannot prove this for any varieties with $p_g(X)>0$ (such as the Fermat sextic fourfold). However, the above argument at least gives a weaker statement concerning $A^2_{AJ}(X)$:

\begin{proposition}\label{A^2} Let $X$ be as in theorem \ref{main3}. Then for any $a,a^\prime\in A^2_{AJ}(X)$, we have
  \[ a\times a^\prime = -a^\prime\times a\ \ \ \hbox{in}\ A^4(X\times X)\ .\]
\end{proposition}

\begin{proof} This is really the same argument as theorem \ref{main3}. We have proven there is a rational equivalence
  \[ \Lambda_{tr}=(\Lambda_{tr})^{\circ N}= Q_1 + Q_2 +\cdots + Q_N\ \ \ \hbox{in}\ A^6(X^4)\ ,\]
      where each $Q_j$ is a composition of $\Lambda_{tr}$ and $\gamma$ in which $\gamma$ occurs at least once. The correspondence $\gamma$ does not act on $A^4(X\times X)$ for dimension reasons (it factors over $A^4(P)$ where $\dim P=3$), and so the $Q_j$ do not act on $A^4(X\times X)$. It follows that
      \[ (\Lambda_{tr})_\ast =\bigl( Q_1 +\cdots + Q_N\bigr){}_\ast =0\colon\ \ A^4(X\times X)\ \to\ A^4(X\times X)\ .\]
  On the other hand, we know from lemma \ref{same} that
       \[ \Lambda_\ast=(\Lambda_{tr})_\ast =0\colon\ \ \ima \Bigl( A^2_{AJ}(X)\otimes A^2_{AJ}(X)\to A^4(X\times X)\Bigr)\ \to\   A^4(X\times X) \ .\]
      This means that for any $a,a^\prime\in A^2_{AJ}(X)$, we have
      \[ \Lambda_\ast (a\times a^\prime) = a\times a^\prime + a^\prime\times a =0\ \ \ \hbox{in}\ A^4(X\times X)\ .\]
      \end{proof}

%%%%%%%%%%%%%%%%%%%%%%%%
%
\section{Applications}
%
%%%%%%%%%%%%%%%%%%%%%%%%

In this section we apply our general result to some Calabi-Yau varieties $X$ of dimension in between $2$ and $5$. First, we give new examples of $\rho$-maximal surfaces. Next, we focus on dimension $3$. Here we give examples of different types. In some cases we prove Voisin's Conjecture as stated in \eqref{conjvois}; in other ones we get the generalization of it on $X \times X \times X \times X$ that appears in Theorem \ref{thm:weak}. Remarkably, one can %check Theorem \ref{main3}
often study the dimension of the $H^n_{tr}(F)$
for a Fermat-type hypersurface $F$ in a certain weighted projective spaces by looking at the (topological) mirror of $F$. Finally, the conjecture is proved in dimension $4$ for the Fermat sextic fourfold and in dimension $5$.

%%%%%%%%%%%%%%%%%%%%%%%%%
%
\subsection{Examples of Dimension $2$}
%
%%%%%%%%%%%%%%%%%%%%%%%%%

\begin{remark} {\em As noted in \cite{moi}, examples of general type surfaces verifying the conditions of Theorem \ref{main3} are contained in the work of Bonfanti \cite{Bonf}. However, many more examples of surfaces verifying the conditions of Theorem \ref{main3} can be found in \cite{BP}. Indeed (as explained to us by Roberto Pignatelli), the ``duals'' (cf. \cite[Section 9]{BP}) of the $14$ families in \cite[Table 2]{BP} are $\rho$--maximal surfaces with $p_g=1$ and $q=0$. Being rationally dominated by a product of curves, these surfaces have finite--dimensional motive.}
\end{remark}

%%%%%%%%%%%%%%%%%%%%%%%%%
%
\subsection{Examples of Dimension $3$ of Fermat type: weak version}
%
%%%%%%%%%%%%%%%%%%%%%%%%%

Let us consider some examples of Calabi-Yau 3folds. One of them is the Fermat quintic $F_5$ in four dimensional projective space, which we work out in full details. We also consider other Fermat type 3folds in weighted projective spaces (for the basics on weighted projective spaces
see e.g. \cite{Dol}).

A different example is taken in \cite{NvG} and is a small resolution $Y'$ of a complete intersection $Y$ of type $(2,2,2,2)$ in seven dimensional projective space. In the Fermat type examples, we are going to show that the dimension of $H^3_{tr}$ is $4$; in the latter example we do not know whether the dimension of $H_{tr}^3(Y')$ is $2$ or $4$. If it were $2$, we could apply our main result and get another example for which Voisin's conjecture holds. If it is $4$, as in the case of $F_5$, we can still deduce something interesting, namely a weak version of Voisin's conjecture thanks to Theorem \ref{thm:weak}.

We start by collecting a useful fact.

\begin{lemma}\label{lem:finite}
Every Fermat hypersurface $\{\sum x_i^d=0\}\subset \mathbb P^n$ %is rationally dominated by a product of curves. In particular it
has finite-dimensional motive.
\end{lemma}
\begin{proof}
A Fermat hypersurface is rationally dominated by curves by the Katsura--Shioda inductive structure \cite{Shi}, \cite[Section 1]{KS}. The analysis of the indeterminacy locus allows to show, cf. \cite{GP}, that this implies that its motive is finite-dimensional.
\end{proof}

Consider now the Fermat quintic hypersurface
$$
 X:= \{ x_0^5+\ldots + x_4^5=0\}\subset \mathbb P^4.
$$
(Later in the paper we will also denote the Fermat quintic hypersurface by $F_5$).
Its Hodge numbers are
$$
 h^{2,1}(X)=101,\ \ h^{1,1}(X)=1=h^{3,0}(X).
$$
Its ``mirror'' $\hat X$ has been constructed explicitely in \cite{GrP, CDGP} as follows.
Inside the quotient $(\mathbb Z/5\mathbb Z)^5/{\textrm {diag}}$ of $(\mathbb Z/5\mathbb Z)^5$ under the natural diagonal action, consider the subgroup $G$ defined by the condition
$$
(a_0,\ldots,a_4)\in G \Longleftrightarrow \sum_i a_i =0.
$$
The subgroup $G$, which is abstractly isomorphic to  $(\mathbb Z/5\mathbb Z)^3$, acts on $X$ and, by \cite[Proposition 4]{Marku} and \cite[Proposition 2]{Roan} the quotient $X/G$ possesses a Calabi-Yau resolution $\hat X$, in other words we have
the following diagram
   \[ \begin{array}[c]{ccc}
         & & X\ \ \\
         & & \downarrow{p}\\
         \hat{X} & \xrightarrow{f}& \ \ X^\prime:=X/G\ .\\
         \end{array}\]
 Notice that the automorphisms $\sigma\in G$ satisfy
     \[\sigma^\ast=\hbox{id}\colon H^{3,0}(X)\to H^{3,0}(X)\ .\]

The variety $\hat X$ turns out to be {\it the mirror} of $X$, see e.g. \cite{Mor, Voisym} for more explanations and details (the analogous construction and the same result hold for any smooth member of the Dwork pencil). In particular its Hodge numbers are
$$
 h^{1,1}(X)=101,\ \ h^{2,1}(X)=1=h^{3,0}(X).
$$

%\begin{theorem}\label{main} Let $X\subset \PP^4(\C)$ be the Fermat quintic
%   \[ (x_0)^5 + \cdots + (x_4)^5=0\ ,\]
%   and let $\hat{X}$ be its mirror. Conjecture \ref{conjvois} is true for $X$ and for $\hat{X}$.
% \end{theorem}
%
% \begin{proof} We need to check that $X$ and $\hat{X}$ verify the hypotheses of Theorem \ref{main3}.
%
 First of all, as observed in Remark \ref{rmk:B(X)}, $X$ verifies $B(X)$ (because it is a projective hypersurface) and has finite--dimensional motive by Lemma \ref{lem:finite}.

We note that $X^\prime$ is a quotient variety $X/G$ for a finite group $G$. As such, there is a well--defined theory of correspondences with rational coefficients for $X^\prime$ (this is because $X^\prime$ has $A^\ast(X^\prime)\cong A_{3-\ast}(X^\prime)$ where $A_\ast$ denotes Chow groups and $A^\ast$ denotes operational Chow cohomology \cite[Example 17.4.10]{F}, \cite[Example 16.1.13]{F}).

Let us denote
  \[ \Gamma:= {}^t \Gamma_f\circ \Gamma_p\ \ \ \in A^3(X\times \hat{X}) \]
  the natural correspondence from $X$ to $\hat{X}$.

Zero--cycles on $X$ and $\hat{X}$ can be related as follows:

 \begin{proposition}\label{A_0} There is an isomorphism of Chow motives
   \[ \Gamma\colon\ \ t_3(X)\cong t_3(\hat{X})\ \ \ \hbox{in}\ \MM_{\rm rat}\ \]
   (with inverse given by ${1\over d}\,{}^t \Gamma$, where $d$ is the order of $G$).
  In particular, the
  homomorphisms
   \[  \begin{split} &f^\ast p_\ast\colon\ \ A^3(X)\ \xrightarrow{}\ A^3(\hat{X})\ ,\\
                              &p^\ast f_\ast\colon\ \ A^3(\hat{X})\ \xrightarrow{}\ A^3(X) \\
                              \end{split}        \]
   are isomorphisms.
   \end{proposition}

 \begin{proof} As we have seen, $X$ satisfies $B(X)$ and has finite--dimensional motive. Moreover, the generalized Hodge conjecture holds for $X$ \cite{Shi2}. The Proposition now follows from the proof of  \cite[Corollary 29(\rom1)]{desult}.
   \end{proof}

 Thanks to Proposition \ref{A_0}, much information can be transported from $X$ to $\hat{X}$, and vice versa. For example, the fact that $B(X)$ holds implies $B(\hat{X})$, because
   \[ h(\hat{X})= t_3(\hat{X})\oplus h(C)\oplus \bigoplus_j \LL(m_j)\ \ \ \hbox{in}\ \MM_{\rm rat}\ ,\]
 %since $h(X)$ and $h(\hat{X})$ differ at most by motives of curves
 where $C$ is a (not necessarily connected) curve. Likewise, the fact that $X$ has finite--dimensional motive implies that $\hat{X}$ has finite--dimensional motive.

% PUT DEFINITIONS OF THE OBJECTS AND DETAILS.

 Alternatively, $B(\hat{X})$ can be proven by invoking the main result of \cite{Tan}, and the finite-dimensionality of the motive of $\hat{X}$  can also be derived from \cite[Example 3.15]{V3} and the fact that $\hat{X}$ is rationally dominated by a product of curves (as $X$ is).

\begin{lemma}
\label{quintic}
Let $X$ be the Fermat quintic in ${\mathbb P}^4$. Then the dimension of $H^3_{tr}(X)$ is $4$.
\end{lemma}

\begin{proof}
%The Fermat quintic has Hodge numbers $(1,101)$. There exists a mirror $F_5^*$ with Hodge numbers $(101,1)$. This means that the dimension of $H^3_{tr}(F_5)$ can be either $2$ or $4$ since the dimension of $H^3_{tr}(F_5^*)$ is at most $4$ and $H^3_{tr}(F_5)$ is isomorphic to $H^3_{tr}(F_5^*)$. Indeed, the mirror of $F_5$ is obtained by taking the quotient of $F_5$ by a suitable group of order $125$ and, after that, taking a crepant resolution. ADD DETAILS OR PROVE A GENERAL PART IN THE A PREVIOUS SECTION.

Take the order $5$ automorphism that permutes the coordinates of ${\mathbb P}^4$. This descends to $X$ and commutes with the elements of the group $G$ of order $125$. Therefore, there exists an order $5$ automorphism of the mirror $\hat{X}$ acting on the four dimensional space of degree $3$ rational cohomology. This space splits into four eigenspaces of such an automorphism, namely
$$
H^3(\hat{X}, {\mathbb Q}) = V(\eta) \oplus V(\eta^2) \oplus V(\eta^3) \oplus V(\eta^4),
$$
where $\eta$ is a primitive fifth root of unity. Up to renaming the primitive root of unity, we can assume that
$H^{3,0} (\hat{X}) \oplus H^{0,3} (\hat{X}) \simeq V(\eta) \oplus V(\eta^4)$, which is not defined over the field of rational numbers. Therefore,  by Proposition \ref{N1max} we have that $\dim H^3_{tr}(\hat{X})=4$. As the  isomorphism of Hodge structures induced by $\Gamma$ yields an isomorphism between  $H^3_{tr}(\hat{X})$ and  $H^3_{tr}({X})$  the Lemma is proved.
\end{proof}

\begin{proposition}
The hypotheses of Theorem \ref{main4} hold for the following Calabi--Yau $3$folds:

\noindent
(1) the Fermat quintic $F_5$ and its mirror;

\noindent
(2) the Fermat hypersurface
   \[ x_0^8 +x_1^8 + x_2^8 +x_3^8 +x_4^2=0 \]
   in weighted projective space $\PP(1^4,4)$ and its mirror;

 \noindent
 (3) the Fermat hypersurface
  \[    x_0^{10} +x_1^{10} + x_2^{10} +x_3^5 +x_4^2=0 \]
   in weighted projective space $\PP(1^3,2,5)$ and its mirror;

\noindent
(4) the Fermat hypersurface
  \[    x_0^{8} +x_1^{8} + x_2^{4} +x_3^4 +x_4^4=0 \]
   in weighted projective space $\PP(1^2,2^3)$ and its mirror;

  \noindent
  (5) the example $Y'$ in \cite{NvG}.
\end{proposition}

\begin{proof}
The claim follows for the Fermat quintic due to Lemma \ref{quintic} and the fact that Fermat hypersurfaces have finite dimensional motive by Lemma \ref{lem:finite}. For examples (2), (3), (4), notice that they are dominated by Fermat hypersurfaces.  The fact that
  \[ \dim H^3_{tr}(X)\le 4 \]
  is established in \cite[Examples 5.3, (c), (d) and Table 4]{KY}. As for the mirror partners, one can directly check that the hypotheses of Theorem \ref{thm:weak}
  are verified.

Let us describe briefly the example in \cite{NvG}. Take the $(2,2,2,2)$ complete intersection in ${\mathbb P}^7$ with homogeneous coordinates $(Y_0:Y_1: Y_2: Y_3: X_0: X_1: X_3: X_4)$ given by
$$
Y_0^2=2(X_0X_1+X_2X_3), \qquad Y_1^2=2(X_0X_2+X_1X_3),
$$
$$
Y_2^2=2(X_0X_3+X_1X_2), \qquad Y_3^2=2(X_0X_1-X_2X_3).
$$

As proved in \cite{NvG}, Proposition 2.8, the dimension of $H^3(X, {\mathbb Q})$ is $4$. Moreover, the Remark on page 69 loc. cit. shows that $Y'$ has finite dimensional motive because there exists a dominant rational map between $C \times E \times E$ and $Y$, where $E$ is the elliptic with complex multiplication of order $4$ and $C$ is a genus $5$ curve. In other words, $Y'$ is dominated by curves, so it has finite dimensional motive.
\end{proof}

\begin{remark} {\em In \cite{BvG} the authors show that there exist families of quintics in four dimensional projective space such that their $H^{3,0}$ is isomorphic to that of the Fermat quintic: see \cite{BvG}, Section 3.3. Also, one of these examples has finite dimensional motive, namely:
$$
x_0^4x_1+x_1^4x_1+x_3^5+x_4^5+x_5^5=0
$$
because the Shioda-Katsura map shows that it is rationally dominated by a product $S \times C$, where $S$ is a Fermat quintic and $C$ is some quintic curve. It follows that Theorem \ref{main4} also applies to this quintic.
} 
\end{remark}

\begin{remark} {\em
Notice that the $N^1$-maximality is also connected to modularity conditions. For instance, Hulek and Verrill in \cite{HV} investigate Calabi-Yau threefolds over the field of rational numbers that contain birational ruled elliptic surfaces $S_j$ for $j=1, \ldots, b$, where $b$ is the dimension of $H^{1,2}(X)$. As they show, this is equivalent to the $N^1$-maximality. Under these assumptions, the $L$-function of $X$ factorizes as a product of the $L$-functions of the base elliptic curves of the birational ruled surfaces and the $L$-function of the weight $4$ modular form associated with the $2$-dimensional Galois representation given by the kernel $U$ of the exact sequence:
$$
0 \rightarrow U \rightarrow H^3_{{e}t}(\overline{X}, {\mathbb Q}_l) \rightarrow \oplus H^3_{{e}t}(\overline{S}_j, {\mathbb Q}_l) \rightarrow 0.
$$

 In \cite{HV}, Section 3, examples of this type of Calabi--Yau varieties are given; however, we do not know whether they have finite dimensional motive.
 }
\end{remark}

%%%%EXAMPLE II

%%%%%%%%%%%%%%%%%%%%%%%%%
%
\subsection{Example of Dimension $3$ of Fermat type: strong version}
%
%%%%%%%%%%%%%%%%%%%%%%%%%
The main result of this paragraph is the following.

\begin{proposition}\label{3fold} Let $X$ be the hypersurface
    \[ \{[x_0:x_1:x_2:x_3:x_4] \ \vert\ x_0^6+x_1^6+x_2^6+x_3^6+x_4^3=0\} \]
   in weighted projective space  ${\mathbb P}^4(1,1,1,1,2)$. Conjecture \ref{conjvois} holds for $X$.
  \end{proposition}

\begin{proof}
It is easy to check that $X$ is a smooth Calabi-Yau variety. Moreover, it can be realized as a degree $3$ finite covering of ${\mathbb P}^3$ branched over the Fermat sextic surface. As such, $X$ has an order $3$ automorphism, say $\tau$. This also shows that is rationally dominated by a product of curves; hence it has finite dimensional motive. It remains to prove the $N^1$-maximality stated in Theorem \ref{main3}. This is proven in \cite[Section 8.3.1, Example 1]{GM}, and also follows readily from \cite[Example 5.3, (b)]{KY}; we propose a more direct proof:

We observe that $X$ can be thought of as the quotient of the degree $6$ Fermat threefold $\{Y_1^6+Y_2^6+Y_3^6+Y_4^6+Y_5^6=0\}$ in four dimensional projective space by the action of the group generated by the automorphism
$$
[Y_1:Y_2:Y_3:Y_4:Y_5] \rightarrow  [Y_1:Y_2:Y_3:Y_4: - Y_5].
$$

The Hodge numbers of $X$ are given by $(h^{1,1}(X), h^{1,2}(X))=(1,103)$. As explained in [KY], the (topological) mirror of $X$ can be described as follows. Take the group
$$
\widehat{G}:= \left\{\left(\varepsilon_6^{i_0}, \varepsilon_6^{i_1}, \varepsilon_6^{i_2}, \varepsilon_6^{i_3}, \varepsilon_3^{i_4}\right): i_0+i_1+i_2+i_3+2i_4 \equiv 0 \,\, \hbox{mod} \,\, 6 \right\}/ H,
$$
where $H$ is a diagonal copy of ${\mathbb Z}/{6\mathbb Z}$ that acts trivially on weighted projective space ${\mathbb P}(1,1,1,1,2)$.

 Let us take into account the polynomials
\begin{equation}
\label{newpencil}
\sum_{I=(i_0,i_1,i_2,i_3,i_4)} C_I x_0^{i_0} x_1^{i_1} x_2^{i_2} x_3^{i_3} x_4^{i_4} + \lambda x_0x_1x_2x_3x_4
\end{equation}
where $\lambda$ varies in ${\mathbb A}^1$, the sum ranges over all solutions of the equation $i_0+i_1+i_2+i_3+2i_4 \equiv 0$ mod $6$ and $C_I$ are generic complex numbers. The vanishing of these polynomials define a pencil of varieties $X'_{\lambda}$ in ${\mathbb P}(1,1,1,1,2)$ that is $\widehat{G}$-invariant. Notice that the members of it are smooth for a generic choice of $\lambda$ because they do not contain the singular point of weighted projective space. A mirror family of $X$ can be found analogously to that of the mirror Fermat quintic by taking the quotient of the pencil \eqref{newpencil} by the group $\widehat{G}$ and, after that, by taking a crepant resolution. Let us denote by $\widehat{X}$ a crepant resolution of $X'_0$.

 Now, let us take into account the order four automorphism $\tau$ of projective space ${\mathbb P}(1,1,1,1,2)$ given by $[x_0, x_1, x_2, x_3, x_4] \rightarrow [x_1, x_2, x_3, x_4, x_0]$. An easy computation shows that $\tau$ belongs to the normalizer of $\widehat{G}$ in the group of automorphisms of ${\mathbb P}(1,1,1,1,2)$. Moreover, there exist complex numbers $C_I$ such that $X_0'$ is invariant with respect to $\tau$. Finally, for such a choice the fixed locus of $\widehat{G}$ is invariant with respect to the ${\tau}$-action because $\tau$ normalizes $\widehat{G}$. Since $\tau$ permutes the homogeneous coordinates of ${\mathbb P}(1,1,1,1,2)$, it extends to all the members of the mirror family, which by definition means that $\tau$ is maximal. Moreover, a direct computation shows that any $\lambda$ is mapped to itself. The space of invariants of $H^{1,2}(X)$ with respect to the $\widehat{G}$-action is thus one-dimensional; hence $\tau$ induces the identity on $H^{1,2}(\widehat{X}) \oplus H^{2,1}(\widehat{X})$. It remains to understand the action induced by $\tau$ on $H^{3,0}(\widehat{X}) \oplus H^{0,3}(\widehat{X})$. For this purpose, we recall that a generator of $H^{3,0}(\widehat{X})$ is a $3$-form on $X$ that is invariant with respect to $\widehat{G}$ - recall that $\widehat{X}$ is a crepant resolution of $X_0'=X/\widehat{G}$. More precisely, this $3$-form can be described as a ratio in which the denominator is $\widehat{G}$-invariant by definition and the numerator is given as follows:
$$
x_0dx_1 \wedge dx_2 \wedge dx_3 \wedge dx_4 - x_1 dx_0 \wedge dx_2 \wedge dx_3 \wedge dx_4 + x_2 dx_0 \wedge dx_1 \wedge dx_3 \wedge dx_4
$$
$$ - x_3 dx_0 \wedge dx_1 \wedge dx_2 \wedge dx_4 +2 x_4 dx_0 \wedge dx_1 \wedge dx_2 \wedge dx_3.
$$

 It is easy to check that this polynomial is mapped to its opposite by the induced action of $\tau$. Therefore, the action on the group $H^{3,0}(\widehat{X}) \oplus H^{0,3}(\widehat{X})$ is the opposite of the identity.

 To recap, the action of $\widehat{\tau}$ on the space $H^3(\hat{X},{\mathbb Q})$ induces a splitting into two eigenspaces of dimension two, one with eigenvalue $+1$ and one with eigenvalue $-1$. This shows the $N^1$-maximality for the Calabi-Yau threefold $\widehat{X}$ and
accordingly, for $X$ because their $H^3_{tr}$'s are isomorphic via an isomorphism of Hodge structures.
\end{proof}

\begin{remark} {\em This example is not new; yet the proof of the $N^1$-maximality is more geometric than those in \cite{KY} and \cite{GM}. In the former reference, the authors prove the maximality by describing two Fermat motives. %In the latter reference, Calabi-Yau varieties with this property are related to special varieties defined over number fields and called {\em attractor varieties}, which appear in IIB string compactifications in Theoretical Physics.
}
\end{remark}

%%%%%%%%%%%%%%%%%%%%%%%%%
%
\subsection{Examples of Dimension $3$ of Borcea-Voisin type: strong version}
%
%%%%%%%%%%%%%%%%%%%%%%%%%

Let $E$ be the elliptic curve given by the equation $y^2=x^3-1$. This curve admits an order three automorphism $h(x,y)=(\omega x, y)$, where $\omega$ is a primitive third root of unit. Now, take $S$ to be a $K3$ surface with an order three automorphism $g$ such that the second cohomology group with rational coefficients splits as the trannscendental $T_S$ and the Neron Severi group such that $H^{2,0}(S) \subseteq T_S$ and the rank of $NS(S)$ is $20$. Moreover, the Neron Severi group coincides with the subspace of invariant classes of $H^2(S, {\mathbb Q})$ with respect to the action of $g$. In particular $g$ is antisymplectic. Such a K3 surface exists as shown in
\cite[p. 280]{BvG}.

The product $S \times E$ admits the order three automorphism $g \times h$. Assume that the action of $g$ on the period of $S$ is given by multiplication by $\omega^2$ (if not, just take the inverse of $g$). Notice that the fixed point locus of $g$ consists of isolated points and (smooth) rational curves.

Denote by $X$ a resolution of the (singular) quotient $S \times E$ by the group generated by the automorphism $g \times h$. By the description of the fixed locus of $g \times h$, the third cohomology group of $X$ with rational coefficients is the invariant part of $H^3(S \times E, {\mathbb Q})$, which is isomorphic to $H^2(S, {\mathbb Q}) \otimes H^1(E, {\mathbb Q})$. To prove the $N^1$-maximality, we check the equivalent condition that $H^{3,0}(S \times E) \oplus H^{0,3}(S \times E)$ is defined over the field of rational numbers. By K\"{u}nneth formula, we have
$$
H^{3,0}(S \times E) \oplus H^{0,3}(S \times E) \simeq  H^{2,0}(S) \otimes H^{1,0}(E) \oplus H^{0,2}(S) \otimes H^{0,1}(E),
$$
$$
H^{2,1}(S \times E) \oplus H^{1,2}(S \times E) \simeq H^{1,1}(S) \otimes H^{1,0}(E) \oplus H^{1,1}(S) \otimes H^{0,1}(E) \oplus
$$
$$
H^{0,2}(S) \otimes H^{1,0}(E) \oplus H^{2,0}(S) \oplus H^{0,1}(E).
$$

The space $H^{3,0}(S \times E) \oplus H^{0,3}(S \times E)$ is defined over the rational field because it can be defined as the subspace of invariants with respect to the action of the isomorphism $(g \times h)^*$ on $H^3(S \times E, {\mathbb Q})$. Indeed, the action of this isomorphism is trivial on $H^{2,0}(S) \otimes H^{1,0}(E) \oplus H^{0,2}(S) \otimes H^{0,1}(E)$. As for $H^{2,1}(S \times E) \oplus H^{1,2}(S \times E)$, the action is by multiplication by $\omega, \, \omega^2, \omega^2, \omega$ on $H^{1,1}(S) \otimes H^{1,0}(E)$, $H^{1,1}(S) \otimes H^{0,1}(E)$, $H^{0,2}(S) \otimes H^{1,0}(E)$, $H^{2,0}(S) \oplus H^{0,1}(E)$ respectively, because the action of $g$ on $H^{1,1}(S)$ is trivial.

%%%%%%%%%%%%%%%%%%%%
%
\subsection{The Fermat 4fold: strong version}
%
%%%%%%%%%%%%%%%%%%%%

We already know that every Fermat hypersurface $\{\sum x_i^d=0\}\subset \mathbb P^n$
has finite-dimensional motive.

%start with an easy fact that we need to apply Theorem \ref{main3}.

As Lefschetz standard conjecture holds for hypersurfaces and the hypothesis
$\wt N^1 H^3 (X)=H^3(X)$ also holds for a 4-dimensional hypersurface, in order to prove
Theorem \ref{main3} we are left with the $\wt N^1$-maximality.

\begin{proposition}\label{sext}
The Fermat sextic fourfold is $\wt N^1$-maximal.
\end{proposition}
\begin{proof}
We will use that
\begin{itemize}
\item[(a)] The Fermat sextic surface $S\subset \mathbb P^3$ is $\rho$-maximal (cf. \cite[Corollary 1]{Beau}.
\item[(b)] The Fermat sextic 4-fold $X\subset \mathbb P^5$ is $\wt N^2$-maximal, i.e.
$\wt N^2 H^4(X)\otimes \mathbb C= H^{2,2}(X)$ (cf. \cite[Corollary 15.11.1]{Mov}).
\end{itemize}
Consider the dominant rational morphism
$$
 \phi : S\times S \dashrightarrow X
$$
It yields a surjective morphism of Hodge structures:
\begin{equation}\label{eq:surj}
\phi_*: H^4_{tr}(S\times S)\twoheadrightarrow H^4_{tr}(X).
\end{equation}
Now $H^4_{tr}(S\times S)\subset H^2_{tr}(S)\times H^2_{tr}(S)$.
By item (a) above
$$
 H^2_{tr}(S)\otimes \mathbb C = H^{2,0}(S)\oplus H^{0,2}(S).
$$
This, together with (\ref{eq:surj}), implies that
$$
 (H^4_{tr}(X)\otimes \mathbb C)\subset H^{4,0}(X)\oplus H^{2,2}(X)\oplus H^{0,4}(X).
$$
By item (b) we see that
  there exists a non--empty Zariski open $\tau\colon U\subset X$ (defined as the complement of the span of the codimension $2$ cycle classes in $H^4(X,\QQ)$) such that
 $H^{2,2}(X)$ maps to $0$ under the restriction map
   \[  \tau^\ast\colon\ \  H^4(X,\C)\ \to\ H^4(U,\C)\ .\]
   This implies that
   \[ \tau^\ast \bigl( H^4_{tr}(X)\otimes\C\bigr) \ \subset\ \tau^\ast \bigl( H^{4,0}(X)\oplus  H^{0,4}(X) \bigr)\ ,\]
   and so the restriction $\tau^\ast \bigl( H^4_{tr}(X)\otimes\C\bigr)$ has dimension at most $2$. On the other hand, by definition of $H^4_{tr}()$, we have that
   \[ \tau^\ast\colon\ \ H^4_{tr}(X)\ \to\ H^4(U) \]
   is an injection. Therefore, we conclude that
   \[ \dim \bigl( H^4_{tr}(X)\otimes\C\bigr)=2\ ,\]
   i.e. $X$ is $N^1$--maximal.
   
   To establish the $\wt{N}^1$--maximality, it remains to show that the inclusion
   \[ \wt{N}^1 H^4(X)\ \subset\ N^1 H^4(X) \]
   is an equality. Here, we again use the dominant rational map $\phi$. The indeterminacy of the map $\phi$ is resolved by the blow--up $\wt{S\times S}$ with center
   $C\times C$ (where $C\subset S$ is a curve). It thus suffices to prove equality
    \[ \wt{N}^1 H^4(\wt{S\times S})=  N^1 H^4(\wt{S\times S})\ . \]
 The blow--up formula gives an isomorphism
  \[ H^4(\wt{S\times S})= H^4(S\times S)\oplus H^2(C\times C)\ ,\]
  and the second summand is entirely contained in $\wt{N}^1$. It thus suffices to prove equality
   \begin{equation}\label{need} \wt{N}^1 H^4({S\times S})\stackrel{??}{=}  N^1 H^4({S\times S})\ . \end{equation}
   This readily follows from the $N^1$--maximality of $S$: indeed, there is a decomposition 
    \[ H^2(S)=T\oplus N\ ,\] 
    where $T:=H^2_{tr}(S)$ is such that $T\otimes\C=H^{2,0}\oplus H^{0,2}$. This induces a decomposition
    \[ H^4(S\times S)= T\otimes T \oplus N\otimes T\oplus T\otimes N \oplus N\otimes N \oplus H^0(S)\otimes H^4(S)\oplus H^4(S)\otimes H^0(S)\ .\]
    All but the first summand are obviously contained in $\wt{N}^1$ (because $D\times S$ satisfies the standard conjecture $B$, for any divisor $D\subset S$). As for the first summand, we note that
    \[  (T\otimes T)_\C\ \ \subset\ H^{4,0}\oplus H^{2,2}\oplus H^{0,4}\ ,\]
    and so 
    \[  N^1 (T\otimes T) = (T\otimes T)\cap F^2 = N^2 (T\otimes T)=\wt{N}^2 (T\otimes T)\ ,\]
    since the Hodge conjecture is true for $S\times S$ \cite[Theorem IV]{Shi}. This proves equality (\ref{need}), and so the $\wt{N}^1$--maximality of $X$ is established.

 %$$
% (H^4_{tr}(X)\otimes \mathbb C)\subset H^{4,0}(X)\oplus H^{0,4}(X)
%$$
%and we are done.
\end{proof}
To finish we observe that all the hypotheses of Theorem \ref{main3} are satisfied for a Fermat sextic fourfold, hence Conjecture \ref{conjvois} holds for it.

%%%%%%%%%%%%%%%%%%%%
%
\subsection{Examples of Dimension $5$}
%
%%%%%%%%%%%%%%%%%%%%

\begin{proposition}[Cynk--Hulek \cite{CH}]\label{ch} Let $E$ be an elliptic curve with an order $3$ automorphism, and let $n$ be a positive integer. There exists a Calabi--Yau variety $X$ of dimension $n$, which is rationally dominated by $E^n$, and which has
 $\dim H^n(X)=2$ if $n$ is even, and $\dim H^n_{tr}(X)=2$ if $n$ is odd.
\end{proposition}

\begin{proof} This is \cite[Theorem 3.3]{CH}. The construction is also explained in \cite[section 5.3]{HKS}.
\end{proof}

\begin{proposition} Let $X$ be a Calabi--Yau variety as in proposition \ref{ch}, of dimension $n\le 5$. Then conjecture \ref{conjvois} is true for $X$.
\end{proposition}

\begin{proof} We check all conditions of Theorem \ref{main3} are satisfied. Point (\rom1) is obvious, as $X$ is rationally dominated by a product of curves. Point (\rom2) is taken care of by Proposition \ref{ch}. Point (\rom3) is proven (in a more general set--up) in \cite[Proof of Corollary 4.1]{smash}.
\end{proof}

\section{Questions}

\begin{question}
Let $F_d$ denote the Calabi--Yau Fermat hypersurface of degree $d$ in $\PP^{d-1}$, i.e.
  \[   x_0^d + x_1^d +\cdots + x_{d-1}^d=0\ .\]
  The variety $F_d$ is $\wt{N}^1$--maximal for $d=4$ and for $d=6$. Are these the only two values of $d$ for which $F_d$ is $\wt{N}^1$--maximal ?

  We suspect this might be the case (by analogy with the $\rho$--maximality of Fermat surfaces in $\PP^3$: as remarked in \cite{Beau}, the only $\rho$--maximal Fermat surfaces are in degree $4$ and $6$), but we have no proof.
  \end{question}

\begin{question} Let $\{X_\lambda\}$ denote the Dwork pencil of Calabi--Yau quintic threefolds
  \[  x_0^5 +x_1^5+ \cdots + x_4^5 + \lambda x_0 x_1 x_2 x_3 x_4 =0\ .\]
  As we have seen, the central fibre $X_0$ has $\dim H^3_{tr}(X_0)=4$. Are there values of $\lambda$ where $\dim H^3_{tr}(X_\lambda)$ drops to $2$ ? Are these values dense in $\PP^1$ ?

 Also, can one somehow prove finite--dimensionality of the motive for non--zero values of $\lambda$ ? (This seems difficult: as noted in \cite[Remark 4.3]{KY}, the varieties $X_\lambda$ are {\em not\/} dominated by a product of curves outside of $\lambda=0$.)
  \end{question}

\begin{nonumberingt} {\em We wish to thank Lie Fu, Bert van Geemen, Hossein Movasati, Roberto Pignatelli and Charles Vial for useful and stimulating exchanges related to this paper.}

\end{nonumberingt}

\end{document}